\documentclass[11pt]{amsart}
\usepackage{epsfig,overpic,comment}
\usepackage[usenames,dvipsnames,svgnames,table]{xcolor}
\usepackage[pagebackref,linktocpage=true,colorlinks=true,linkcolor=Blue,citecolor=BrickRed,urlcolor=RoyalBlue]{hyperref}
\usepackage[msc-links,abbrev,backrefs]{amsrefs} 

\newtheorem{thm}{Theorem}[section]
\newtheorem{lem}[thm]{Lemma}
\newtheorem{cor}[thm]{Corollary}

\theoremstyle{definition}

\newtheorem{note}[thm]{Note}

\newcommand{\R}{\mathbf{R}}

\newcommand{\ol}{\overline}
\newcommand{\C}{\mathcal{C}}

\renewcommand{\S}{\mathbf{S}}
\renewcommand{\l}{\langle}
\renewcommand{\r}{\rangle}

\title[]{Torsion of locally convex curves}

\author{Mohammad Ghomi}
\address{School of Mathematics, Georgia Institute of Technology,
Atlanta, GA 30332}
\email{ghomi@math.gatech.edu}
\urladdr{www.math.gatech.edu/$\sim$ghomi}

\date{\today \,(Last Typeset)}
\thanks{Research of the  author was supported in part by NSF Grants DMS--1308777 and DMS -1711400.}

\begin{document}

\begin{abstract}
We show that the torsion of any simple closed curve $\Gamma$ in Euclidean 3-space changes sign at least $4$ times provided that it is star-shaped and locally convex with respect to a point $o$ in the interior of its convex hull. The latter condition means that through each point $p$ of $\Gamma$ there passes a  plane $H$, not containing $o$, such that a neighborhood of $p$ in $\Gamma$ lies on the same side  of $H$ as does $o$. This generalizes the four vertex theorem of Sedykh for convex space curves. Following Thorbergsson and Umehara, we reduce the proof to the result of Segre on inflections of spherical curves, which is also known as Arnold's tennis ball theorem.
\end{abstract}

\maketitle

\section{Introduction}
In 1992 Sedykh \cite{sedykh:originalvertex, sedykh:vertex} generalized the classical four vertex theorem of planar curves by showing that the torsion of any closed space curve  vanishes at least $4$ times, if it lies on a convex surface, see Note \ref{note:4vertex}. Recently the author \cite{ghomi:rosenberg} extended Sedykh's theorem to curves which lie on a \emph{locally} convex simply connected surface. In this work we prove another generalization of Sedykh's theorem which does not require the existence of any underlying surface for the curve. 

To state our result, let us recall that 
a set $X$ in Euclidean space $\R^3$ is \emph{star-shaped} with respect to a point $o$ if no ray emanating from $o$ intersects $X$ in more than one point. Further let us say that $X$ is \emph{locally convex} with respect to $o$ if through  every point $p$ of $X$ there passes a plane $H$, not containing $o$, such that a neighborhood of $p$ in $X$ lies on the same side of $H$ as does $o$. For instance, the boundary of any convex body in $\R^3$ is both star-shaped and locally convex with respect to any of its interior points. In this paper we show:

\begin{thm}\label{thm:main}
Let $\Gamma\subset\R^3$ be a simple closed $\C^3$ immersed curve with nonvanishing curvature. Suppose that $\Gamma$ is star-shaped and locally convex with respect to a point  $o$ in the interior of its convex hull. Then the torsion of $\Gamma$ changes sign at least $4$ times.
\end{thm}

In particular,  if $\Gamma$ lies on  the boundary of a convex body, then it immediately follows that $\Gamma$ has at least $4$ points of vanishing torsion, which is Sedykh's result. The above theorem also generalizes a similar result of Thorbergsson and Umehara \cite[Thm. 0.2]{thorbergsson&umehara}, see Note \ref{note:TU2}.

The general strategy for proving Theorem \ref{thm:main} hinges on the fact that at a point of nonvanishing torsion $p$, the torsion $\tau$ of $\Gamma$ is positive (resp. negative) if and only if $\Gamma$ crosses its osculating plane at $p$ in the direction (resp. opposite direction) of the binormal vector $B(p)$. To exploit this phenomenon, we project $\Gamma$ into a sphere $S$ centered at $o$ to obtain a simple closed  curve $\ol\Gamma$ which contains $o$ in its convex hull. By a result of Segre \cite{segre:tangents, weiner:inflection, ghomi:verticesC}, also known as Arnold's tennis ball theorem, $\ol\Gamma$ has at least $4$ inflections $\ol p_i$. It turns out that  the osculating planes $\ol\Pi_{\ol p_i}$ of $\ol\Gamma$ coincide with the osculating planes $\Pi_{p_i}$ of $\Gamma$ at $p_i$, where $p_i$ are the preimages of $\ol p_i$, see Figure \ref{fig:projection}(a). Further, the local convexity assumption will  ensure that the binormal vectors of these planes are parallel, i.e., $\ol B(\ol p_i)=B(p_i)$. So the local position of $\Gamma$ with respect to $\Pi_{p_i}$ mirrors that of $\ol\Gamma$ with respect to $\ol\Pi_{\ol p_i}$. After a perturbation of $o$, we may assume that $\tau(p_i)\neq 0$ and $\ol p_i$ are \emph{genuine} inflections, i.e.,  the geodesic curvature of $\ol\Gamma$ changes sign at $\ol p_i$. Further, it is easy to see that at every pair of consecutive genuine inflections, $\ol\Gamma$ crosses its osculating planes in opposite directions with respect to $\ol B$, because $\ol B$ can never be orthogonal to $S$, see Figure \ref{fig:projection}(b).
Thus it follows that $\tau$ changes sign at least once within each of the $4$ intervals of $\Gamma$ determined by $p_i$.
 \begin{figure}[h]
   \centering
    \begin{overpic}[height=1.75in]{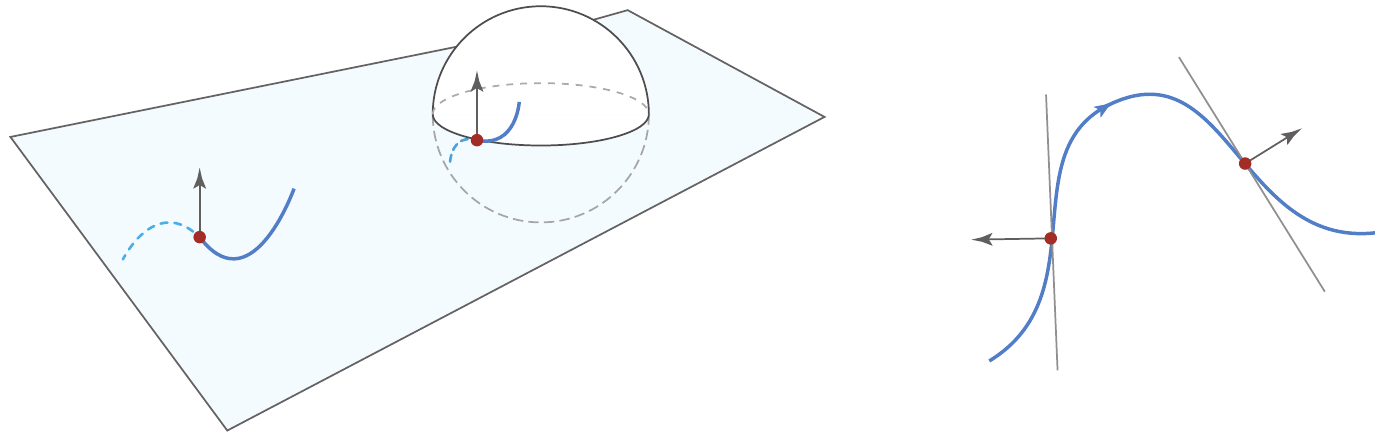}
         \put(13.5,12.5){\small $p_i$}
         \put(13,20){\small $B(p_i)$}
         \put(22,16){\small $\Gamma$}
          \put(34.5,19){\small $\ol p_i$}
         \put(33,27){\small $\ol B(\ol{p}_i)$}
         \put(38.5,23){\small $\ol\Gamma$}
         \put(30,2){\small (a)}
         \put(81,2){\small (b)}
         \put(13.5,6){$\Pi_{p_i}=\ol\Pi_{\ol{p}_i}$}
                    \put(81,25.5){\small $\ol\Gamma$}
            \put(94,23){\small $\ol{B}$}
           \put(68,14){\small $\ol{B}$}

    \end{overpic}
    \caption{}\label{fig:projection}
\end{figure}

The above approach for studying the sign of torsion is due to Thorbergsson and Umehara \cite[p. 240]{thorbergsson&umehara}, who in turn attribute the spherical projection technique to Segre \cite[p. 258]{segre1968alcune}; however, the method of Thorbergsson and Umehara does not quite lead to a generalization of Sedykh's theorem to a class of nonconvex curves, as we describe in Note \ref{note:TU} below. 

Next we discuss some examples which validate Theorem \ref{thm:main}.  Figure \ref{fig:curls}(a) shows a curve which is star-shaped and locally convex, but is not convex. Thus Theorem \ref{thm:main} is  indeed a nontrivial generalization of Sedykh's result.  Figure \ref{fig:curls}(b), which shows a torus curve of type $(1, 7)$, demonstrates that the star-shaped assumption by itself would not be sufficient in Theorem \ref{thm:main}. Indeed all torus curves of type $(1,n)$ are star-shaped with respect to any point on their axis of symmetry, and Costa \cite{costa:twisted} has shown that, for $n\geq 2$, they may be realized with nonvanishing torsion if the underlying torus is sufficiently thin. Finally Figure \ref{fig:curls}(c) shows that the local convexity by itself is not sufficient either. This figure depicts a spherical curve with only two extrema of geodesic curvature and hence only two sign changes of torsion (the torsion of a spherical curve vanishes precisely when its geodesic curvature has a local extremum).

 \begin{figure}[h]
   \centering
    \begin{overpic}[height=1in]{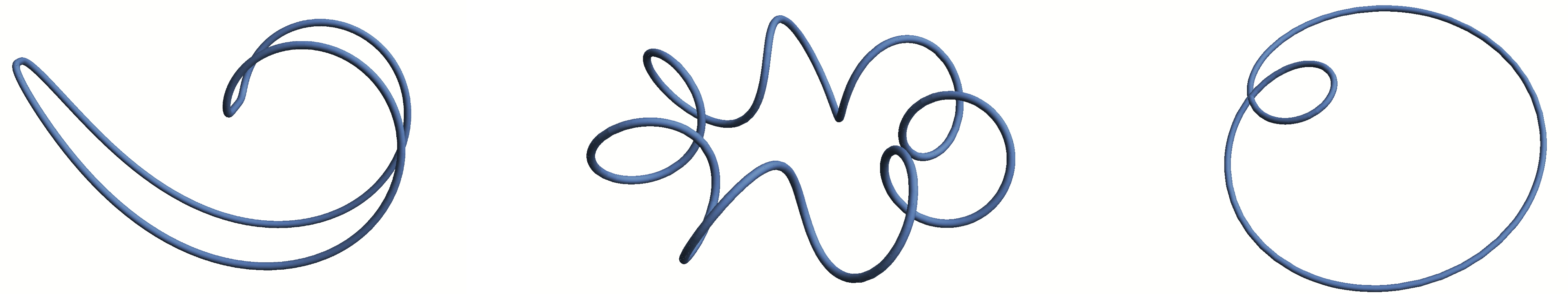}
    \put(10,-2){\small (a)}
    \put(45,-2){\small (b)}
    \put(78,-2){\small (c)}
        \end{overpic}
    \caption{}\label{fig:curls}
\end{figure}

Four vertex theorems have had a multifaceted and interesting history, with unexpected applications, since Mukhopadhyaya proved the first version of this phenomenon in 1909. For extensive background and more references see \cite{ghomi:rosenberg, gluck:notices, ovsienko&tabachnikov}. In particular see \cite{bray&jauregui} for some recent applications in General Relativity, and \cite{sedykh1996, sedykh1997} for discrete versions.

\begin{note}\label{note:4vertex}
The classical four vertex theorem  states that the curvature of any simple closed curve in $\R^2$ has at least $4$ critical points, which are called vertices.  A point of the curve is a vertex if and only if the osculating circle at that point has contact of order $3$ with the curve. Consequently, the geodesic curvature of simple closed curves on the sphere also satisfies the four vertex theorem, because the stereographic projection preserves circles. Further note that the plane which contains the osculating circle of a spherical curve is actually the osculating plane of the curve. Thus at a vertex, a spherical curve has contact of order $3$ with its osculating plane, which means that the torsion at that point vanishes. Hence all simple closed spherical curves have at least $4$ points of vanishing torsion. Sedykh generalized this result to curves lying on any convex surface. It is in this sense that Sedykh's theorem, and more generally our Theorem \ref{thm:main}, are extensions of the classical four vertex theorem.
\end{note}

\begin{note}\label{note:TU}
For convex space curves, a refinement of Sedykh's four vertex theorem is proved by  Thorbegsson and Umehara  in \cite[Thm. 0.1]{thorbergsson&umehara}. Furthermore,  these authors \cite[Thm. 0.2]{thorbergsson&umehara} obtain  a $4$ vertex result for space curves $\Gamma\subset \R^3$ which are star-shaped with respect to a point $o$ in the interior of their convex hull, have no tangent line passing through $o$, and further satisfy the property that for all points $p\in\Gamma$ the angle between the principal normal $N(p)$ of $\Gamma$ and the position vector $p-o$ is obtuse, i.e., 
\begin{equation}\label{eq:N}
\langle p-o, N(p) \rangle <0.
\end{equation}
They claim that this result implies Sedykh's theorem, because a ``convex space curve $\gamma$ satisfies the conditions in Theorem 0.2" \cite[p. 230]{thorbergsson&umehara}; however, this is not the case. Indeed there exists a simple closed curve  which lies on the boundary of a convex body,  but does not satisfy the condition \eqref{eq:N} for any point $o$; see Figure \ref{fig:apple}.
  \begin{figure}[h]
   \centering
    \begin{overpic}[height=1.1in]{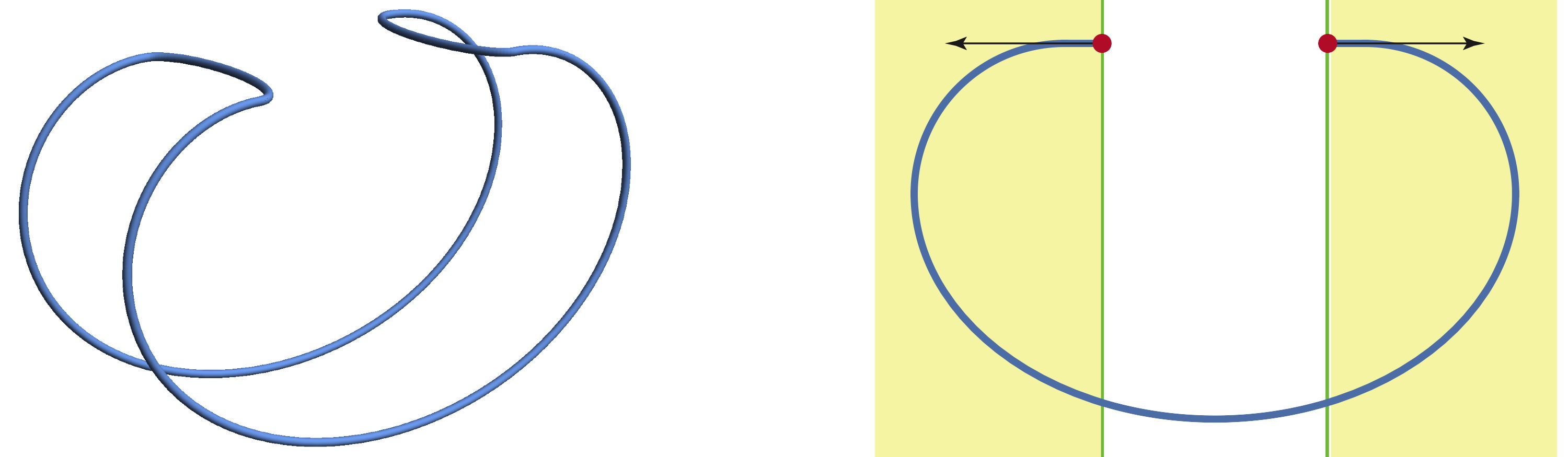}
    \put(71.5,26){\small$p_1$}
    \put(80,26){\small$p_2$}
     \put(57,1.5){\small$H_1^+$}
      \put(94,1.5){\small$H_2^+$}
    \end{overpic}
    \caption{}\label{fig:apple}
\end{figure}
The left diagram here depicts the curve,  and the right diagram shows its projection into its plane of symmetry. Let $p_i$, $i=1$, $2$ be the intersections of the curve with it symmetry plane. Note that at these points the principal normals $N(p_i)$ are antiparallel. Let $H_i$  be the planes orthogonal to $N(p_i)$ which pass through $p_i$, and $H_i^+$ be the corresponding (closed) half-spaces into which $N(p_i)$ point. Note that each $H_i^+$ consists of the set of all points $o$ such that 
$\l p_i-o, N(p_i) \r \leq 0$. But $H_1$ and $H_2$ are disjoint. So there exists no point $o$ with respect to which $\Gamma$ can satisfy condition \eqref{eq:N}. Thus Theorem 0.2 in \cite{thorbergsson&umehara} does not imply Sedykh's theorem.
\end{note}

\begin{note}\label{note:TU2}
The result of Thorbergsson and Umehara  \cite[Thm. 0.2]{thorbergsson&umehara}  mentioned in the previous note is a special case of Theorem \ref{thm:main}. Indeed let $H_p$ be the rectifying plane of $\Gamma$ at $p$, i.e., the plane which passes through $p$ and is orthogonal to the principal normal $N(p)$. Then \eqref{eq:N} implies that $N(p)$ points to the side of $H$ which contains $o$. Consequently a neighborhood of $p$ in $\Gamma$ lies on the same side as well, which establishes the local convexity of $\Gamma$.
\end{note}

\section{Basic Notation and Terminology}
Throughout this work we assume that $\Gamma$, $o$ are as in Theorem \ref{thm:main}. In particular $\Gamma$ has nonvanishing curvature (so that its torsion is well defined). For convenience we also assume that $o$ is the origin of $\R^3$. Further we let $\ol\Gamma$ denote the radial projection of $\Gamma$ into the unit sphere $\S^2$ centered at $o$. For every point $p\in\Gamma$, $\ol p:=p/\|p\|$ denotes the corresponding point of $\ol\Gamma$. We assume that $\Gamma$ is oriented, and let $T$, $N$, $B:=T\times N$, denote the corresponding unit tangent, principal normal, and the binormal vectors of $\Gamma$ respectively. For every point $p\in\Gamma$ there exists a $(\C^3)$ unit speed parametrization $\gamma\colon (-\epsilon,\epsilon)\to \Gamma$ with $\gamma(0)=p$ such that $\gamma'(0)=T(p)$. Then $N(p):=\gamma''(0)/\|\gamma''(0)\|$, and  the torsion of $\Gamma$ is given by 
$$
\tau(p):=\frac{\langle \gamma'''(0),B(p)\rangle}{\|\gamma''(0)\|}.
$$
 The osculating plane $\Pi_p$ of $\Gamma$ at $p$ is the plane which passes through $p$ and is orthogonal to $B(p)$. We let $\ol\gamma$, $\ol T$, $\ol N$, $\ol B$,   $\ol \Pi$ denote the corresponding quantities for $\ol\Gamma$. More specifically, $\ol\gamma:=\gamma/\|\gamma\|$, $\ol T:=\ol\gamma'/\|\ol\gamma'\|$, $\ol B:=\ol\gamma'\times\ol\gamma''/\|\ol\gamma'\times\ol\gamma''\|$, and $\ol N:=\ol B\times \ol T$. In particular note that these quantities are well defined, since by assumption through each point of $\Gamma$ there passes a local support plane of $H$ not containing $o$. Consequently the tangent lines of $\Gamma$ do not pass through $o$, since they  lie in $H$. So $\|\ol\gamma'\|\neq 0$.
 Further note that $\ol\Gamma$ inherits its orientation from $\Gamma$.
An \emph{inflection point} of $\ol\Gamma$ is a point where the geodesic curvature $\ol k$ of $\ol\Gamma$ in $\S^2$ vanishes. Here we define $\ol k$ with respect to the conormal vector $\ol n(\ol p):=\ol p\times \ol T(\ol p)$:
$$
\ol k(\ol p):=\langle \ol N(\ol p),\ol n(\ol p)\rangle.
$$
We say that an inflection point $\ol p$ is genuine if $\ol k$ changes sign at $\ol p$.

\section{Osculating Planes and Inflections}
A key part of the proof of Theorem \ref{thm:main}, which facilitates its reduction to Segre's tennis ball theorem, is the following observation. The first part of this lemma is trivial, the second part goes back to Segre, and the third part is a consequence of our local convexity assumption. 
\begin{lem}\label{lem:osculate}
Let $\ol p$ be an inflection point of $\ol\Gamma$. Then 
\begin{enumerate}
\item[(i)] $o\in\ol\Pi_{\ol p}$, 
\item[(ii)] $\ol\Pi_{\ol p}=\Pi_p$, 
\item[(iii)] $B(p)= \ol B(\ol p)$.
\end{enumerate}
\end{lem}
\begin{proof}
The argument is presented in three parts corresponding to each of the items enumerated above. Here for any pair of vector $v$, $w$, we use the notation $v\parallel w$ to indicate that $v=\lambda w$ for some $\lambda>0$.

(\emph{i}) 
The point $\ol p$ is an inflection if and only if $\ol N(\ol p)$ is orthogonal to $\S^2$, or
\begin{equation}\label{eq:olN}
\ol N(\ol p)=-\ol p.
\end{equation}
 Thus $o=\ol p+\ol N(\ol p)\in\ol \Pi_{\ol p}$, as claimed. 

(\emph{ii}) 
Since $\ol\Pi_{\ol p}$ passes through $\ol p$ and $o$, it contains $p$ as well. So it suffices to check that $\ol\Pi_{\ol p}$ and $\Pi_p$ are parallel, or that $\ol B(\ol p)$ is orthogonal to $\Pi_p$. To this end let $\gamma\colon (-\epsilon, \epsilon)\to\Gamma$ be a  local parametrization with $\gamma(0)=p$ and $\gamma'(0)\parallel T(p)$. Then $\ol\gamma:=\gamma/\|\gamma\|$ yields a local parametrization for $\ol\Gamma$ with $\ol\gamma(0)=\ol p$ and $\ol\gamma'(0)\parallel \ol T(\ol p)$.  We need to check that $\ol B(\ol p)$ is orthogonal to both $\gamma'(0)$ and $\gamma''(0)$. Simple computations show that
\begin{equation}\label{eq:olgamma1}
\ol\gamma\times \ol\gamma'=\frac{\gamma\times \gamma'}{\|\gamma\|^2},
\end{equation}
\begin{equation}\label{eq:olgamma2}
\ol\gamma\times\ol\gamma''=\frac{(\gamma\times\gamma'')\|\gamma\|^2-2(\gamma\times\gamma')\l\gamma,\gamma'\r}{\|\gamma\|^4}.
\end{equation}
Using \eqref{eq:olN} and \eqref{eq:olgamma1}, we have
\begin{equation}\label{eq:olB}
\ol B(\ol p)=\ol T(\ol p)\times \ol N(\ol p)=\ol p\times \ol T(\ol p)=\ol\gamma(0)\times \ol\gamma'(0)
=\frac{\gamma(0)\times \gamma'(0)}{\|\gamma(0)\|^2}.
\end{equation}
Thus
$
\langle \gamma'(0), \ol B(\ol p)\rangle=0.
$
Next, to compute $\l\gamma''(0), \ol B(\ol p)\r$, we may assume that $\ol\gamma$ has unit speed. Then $\ol N(\ol p)\parallel \ol\gamma''(0)$ and \eqref{eq:olN} yields that 
$$
\ol\gamma(0)\times \ol\gamma''(0)\parallel  \ol p\times \ol N(\ol p)\parallel \ol p\times(-\ol p)=0.
$$
Consequently, \eqref{eq:olgamma2} yields that $\gamma(0)\times\gamma''(0)=\alpha\; \gamma(0)\times\gamma'(0)$, for some constant $\alpha$. Now
using \eqref{eq:olB} we have, 
$$
\|\gamma(0)\|^2\l\gamma''(0), \ol B(\ol p)\r=\l\gamma'(0),\gamma''(0)\times\gamma(0)\r=
\alpha\l\gamma'(0),\gamma'(0)\times\gamma(0)\r=0,
$$
as desired. 

(\emph{iii})
We may assume that $\gamma$ has unit speed. Then $N(p)\parallel \gamma''(0)$.
Consequently,
$$
B(p)=T(p)\times N(p)\parallel \gamma'(0)\times \gamma''(0).
$$
This together with \eqref{eq:olB} shows that $B(p)\parallel\ol B(\ol p)$ if and only if
$$
\gamma'(0)\times\gamma''(0)\parallel \gamma(0)\times \gamma'(0).
$$
By assumption there exists a local support plane $H$ of $\Gamma$ passing through $p$.   Let $L$ be the line given by $H\cap \Pi_p$ ($L$ is well defined because $o\in \Pi_p$ by parts $(i)$ and $(ii)$ above, but $o\not\in H$ by assumption; thus $H\neq\Pi_p$). Let $L^+$ denote the side of $L$ in $\Pi_p$ which contains $o$. Then $N(p)$ points into $L^+$, see Figure \ref{fig:planes}; because by assumption $\Gamma$ lies locally on the side of $H$ which contains $o$, and thus $N(p)$ must point into this side as well. 
  \begin{figure}[h]
   \centering
    \begin{overpic}[height=1.8in]{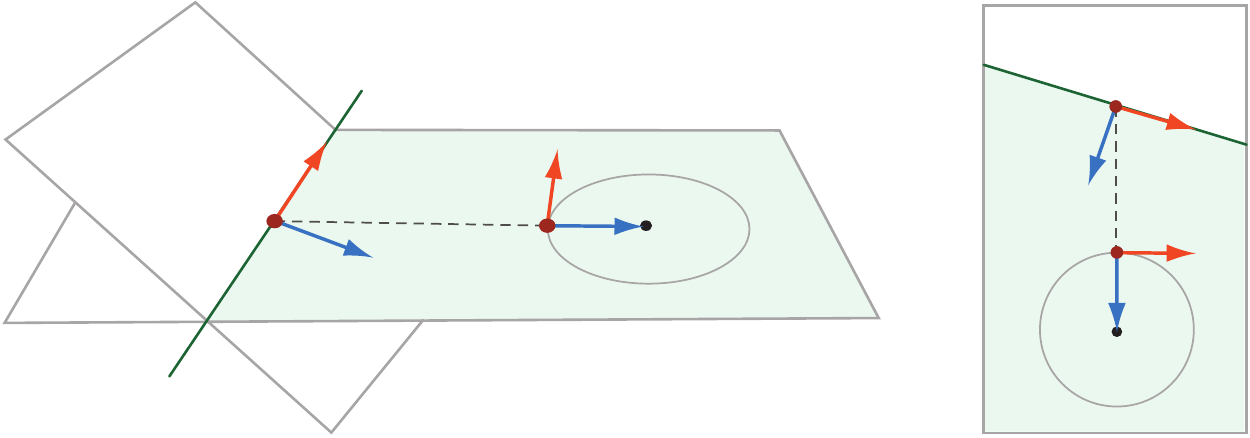}
    \put(52.5,16){\small$o$}
    \put(42,14){\small$\ol{p}$}
    \put(21,14){\small$p$}
    \put(11,3){\small$L$}
     \put(6,23){\small$H$}
     \put(63,11){\small$L^+$}
     \put(76,29){\small$L$}
      \put(19,23.5){\small$T(p)$}
       \put(43,24){\small$\ol T(\ol p)$}
       \put(27,11.5){\small$N(p)$}
       \put(50,13){\small$\ol N(\ol p)$}
       \put(0,11){\small$\Pi_{p}=\ol\Pi_{\ol{p}}$}
       \put(89, 6.5){\small$o$}
        \put(89, 16.5){\small$\ol p$}
         \put(89, 28.5){\small$p$}
         \put(79, 1){\small$L^+$}
         \put(93.5, 26.5){\small$T(p)$}
         \put(94, 16){\small$\ol T(p)$}
         \put(80.5, 22){\small$N(p)$}
         \put(82, 9){\small$\ol N(p)$}
       \end{overpic}
    \caption{}\label{fig:planes}
\end{figure}
 Further, $L$ is tangent to $\Gamma$ since $H$ is tangent to $\Gamma$. Thus $N(p)\perp L$. Consequently 
 $$
 L^+=\{ x\in \Pi_p\mid \l\ x-p,N(p)\r\geq 0\}.
 $$  
 So since $o$ lies in the interior of $L^+$, $\l o-p, N(p)\r>0$, or
 $
\langle N(p), p\rangle <0,
$
 which  yields
\begin{equation}\label{eq:a}
 \langle \gamma''(0), \gamma(0)\rangle <0.
\end{equation}
 Since $\Pi_p$ passes through $o$,  $\{\gamma(0),\gamma'(0),\gamma''(0)\}$ is linearly dependent.
 Further, since we are assuming that $\gamma$ has unit speed, $\gamma'\perp\gamma''$. Thus 
 $$
 \gamma(0)=a \gamma'(0)+b\gamma''(0),
 $$
 where $b=\langle \gamma''(0), \gamma(0)\rangle/\|\gamma''(0)\|^2<0$ according to \eqref{eq:a}.
 So
 $$
 \gamma(0)\times\gamma'(0)=-b\;\gamma'(0)\times\gamma''(0),
 $$
 which completes the proof.
\end{proof}

\section{Maximum Principles for Torsion and Geodesic Curvature}

Here we collect the facts we need concerning the relation between the sign of torsion of a space curve, and its relative position with respect to its osculating plane. Further we discuss the corresponding facts for the geodesic curvature of spherical curves, which will be used in the proof of our main result in the next section.

We start with torsion. Here by the region \emph{above} the osculating plane we mean the (closed) half-space into which the binormal vector $B$ points, and the region \emph{below} will be the other half-space. Since we assume that $\Gamma$ is oriented, for every pair of points $p$, $q$ of $\Gamma$, there is a unique choice of a segment with initial point $p$ and final point $q$ which we denote by $[p,q]$. The interior of this segment will be denoted by $(p,q)$.

\begin{lem}[Lem. 6.12, \cite{ghomi:rosenberg}]\label{lem:maxprincipletorsion}
Suppose that $\tau\geq  0$ (resp. $\tau\leq 0$) on a segment $[p, q]$ of $\Gamma$. Then, near $p$,  the segment lies above (resp. below) the osculating plane of $\Gamma$ at $p$.\qed
\end{lem}

The above lemma quickly yields the following converse:

\begin{cor}\label{cor:maxprincipletorsion}
Suppose that a segment $[p,q]$ of $\Gamma$ lies above its osculating  plane at $p$ (resp. $q$) and does not lie completely in the osculating plane. Then $\tau >0$ (resp. $\tau<0$) at some point of $[p,q]$.
\end{cor}
\begin{proof}
First suppose that $[p,q]$ lies above its osculating plane at $p$ and assume, towards a contradiction, that $\tau\leq 0$ on $[p,q]$. Then $[p,q]$ lies below $\Pi_p$ by Lemma \ref{lem:maxprincipletorsion}. Thus $[p,q]$ must lie entirely in $\Pi_p$ which is a contradiction. The case where $[p,q]$ lies above its osculating plane at $q$, also follows from Lemma \ref{lem:maxprincipletorsion}, once we switch the orientation of $[p,q]$ and observe that this switches the direction of $B$, but does not effect the sign of $\tau$.
\end{proof}

Similarly, here are the  facts concerning geodesic curvature which we need. For every $\ol p\in\ol\Gamma$, let $C_{\ol p}$ denote the great circle in $\S^2$ which is tangent to $\ol\Gamma$ at $\ol p$. By the region \emph{above} $C_{\ol p}$ we mean the (closed) hemisphere into which the conormal vector $\ol n(\ol p):=\ol p\times \ol T(\ol p)$ points, and the other hemisphere will be referred to as the region \emph{below} $C_{\ol p}$.

\begin{lem}[Lem. 2.1, \cite{ghomi:verticesC}]\label{lem:maxprinciplecurvature}
Suppose that  a segment $[\ol p, \ol q]$ of $\ol\Gamma$ lies above (resp. below) the tangent great circle $C_{\ol p}$. Then either  $[\ol p, \ol q]$ is a part of $C_{\ol p}$, or else $\ol k > 0$ (resp. $< 0$) at some point of $[\ol p, \ol q]$.\qed
\end{lem}

Now we quickly obtain:

\begin{cor}\label{cor:maxprinciplecurvature}
Suppose that $\ol k>0$ on the interior of a segment $[\ol p,\ol q]$ of $\ol\Gamma$. Then, near $\ol p$ and $\ol q$,  $[\ol p, \ol q]$ lies above $C_{\ol p}$ and $C_{\ol q}$ respectively, and does not coincide with them.
\end{cor}
\begin{proof}
If $[\ol p,\ol q]$ lies locally below  $C_{\ol p}$, then $[\ol p, \ol q]$  coincides with $C_{\ol p}$ near $\ol p$, by Lemma \ref{lem:maxprinciplecurvature}. So $[\ol p, \ol q]$ lies locally above $C_{\ol p}$ as claimed. The same argument may also be applied to $\ol q$ after switching the orientation of $\ol\Gamma$. Finally, since $\ol k\neq 0$ on $(p,q)$, these circles cannot coincide with $\ol\Gamma$ near $\ol p$ and $\ol q$.
\end{proof}

Next, we link the last two corollaries together by recording that if $\ol p$ is an inflection point of $\ol\Gamma$, then the region above $C_{\ol p}$ corresponds to the region above $\Pi_{p}$. Indeed recall that if $\ol p$ is an inflection, then $\ol N(\ol p)=-\ol p$ as we discussed in the proof of Lemma \ref{lem:osculate}. Thus
$$
\ol n(\ol p)=\ol p\times \ol T(\ol p)=-\ol N(\ol p)\times \ol T(\ol p)=\ol T(\ol p)\times \ol N(\ol p)=\ol B(\ol p).
$$

So Lemma \ref{lem:osculate} quickly yields:

\begin{lem}\label{lem:above}
At an inflection point  $\ol p$ of $\ol\Gamma$, the region above the great circle $C_{\ol p}$ in $\S^2$ coincides with the hemisphere which lies above the osculating plane $\Pi_p$.\qed
\end{lem}

\section{Proof of the Main Result}

Before proving our main theorem, we require the following technical fact which shows that the local convexity and star-shaped  properties of $\Gamma$ are stable under small perturbations of $o$.

\begin{lem}\label{lem:oprime}
There exists an open neighborhood $U$ of $o$ such that $\Gamma$ is star-shaped and locally convex with respect to every point $o'$ in $U$.
Further, we may choose $o'$ so that only finitely many osculating planes of $\Gamma$ pass through $o'$.
 \end{lem}
\begin{proof}
The local convexity assumption ensures that the tangent lines of $\Gamma$ do not pass through $o$. Thus $\ol\Gamma$ is a $\C^3$ immersed curve. Further, the star-shaped assumption ensures that $\ol\Gamma$ is embedded. Now since embeddings of compact manifold are open in the space of $\C^1$ mappings, it follows that projections of $\Gamma$ into unit spheres centered at $o'\in U$ are embedded  as well, for some open neighborhood $U$ of $o$. So the star-shaped property is preserved under small perturbations of $o$. Next we  check that the local convexity assumption is preserved as well. To this end it suffices to show that the local support planes of $\Gamma$ may be chosen continuously, for then the support planes cannot get arbitrarily close to $o$, and hence $\Gamma$ remains locally convex with respect to all points of $U$, assuming $U$ is sufficiently small. To see that the local support planes may be chosen continuously, see \cite[Sec 3.1]{ghomi:stconvex}.
Finally, using Sard's theorem, we may choose a point $o'$ in $U$ such that only finitely many osculating planes of $\Gamma$ pass through $o'$: consider  $f\colon\Gamma\times\R^2\to\R^3$ given by $f(p, t, s)= p+t\, T(p) +s\, N(p)$, and let $o'$ be a regular value of $f$. 
\end{proof}

We also need the following refinement of the tennis ball theorem. Recall that an inflection point $\ol p$ of $\ol\Gamma$ is \emph{genuine}, provided that the geodesic curvature $\ol k$ of $\ol\Gamma$ changes sign at $\ol p$.

\begin{lem}[Thm. 1.2, \cite{ghomi:verticesC}]\label{lem:tennisball}
Suppose that $\ol\Gamma$ has at most  finitely many inflections. Then at least $4$ of these inflections must be genuine.\qed
\end{lem}

Finally we are ready to establish our main result:

\begin{proof}[Proof of Theorem \ref{thm:main}]
By Lemma \ref{lem:oprime}, after replacing $o$ by a nearby point, we may assume that at most finitely many osculating planes of $\Gamma$ pass through $o$. By Lemma \ref{lem:osculate}, osculating planes of $\ol\Gamma$ at its inflections pass through $o$ and coincide with the osculating planes of $\Gamma$. Thus  $\ol\Gamma$ has now at most finitely many inflections. 
Consequently, by the refinement of the tennis ball theorem, Lemma \ref{lem:tennisball},  $\ol\Gamma$ has at least $4$ genuine inflections. 

Let $\ol p_0$, $\ol p_1$ be a pair of  genuine inflections of $\ol\Gamma$ such that the oriented segment 
$[\ol p_0, \ol p_1]$   has no genuine inflections in its interior $(\ol p_0, \ol p_1)$. There are at least $4$ such intervals with pairwise disjoint interiors. Thus, to complete the proof, it suffices to show that $\tau$ changes sign on $(p_0, p_1)$. 

We may assume that $\ol k\geq 0$ on $(\ol p_0,\ol p_1)$, after switching the orientations of $\Gamma$ and $\ol\Gamma$ if necessary. Then $\ol k>0$ near $\ol p_i$ on $(\ol p_0,\ol p_1)$, since  $\ol\Gamma$ has only finitely many inflections. Consequently,  by the maximum principle for geodesic curvature, Corollary \ref{cor:maxprinciplecurvature}, $[\ol p_0, \ol p_1]$ lies locally above its tangent great circles at $\ol p_i$ and does not coincide with them near $\ol p_i$.
This in turn implies, by Lemma \ref{lem:above},
 that 
$[p_0, p_1]$ lies locally above its osculating planes at $p_i$, and does not coincide with them near $p_i$. Thus, by  the maximum principle for torsion, Corollary \ref{cor:maxprincipletorsion}, $\tau$ changes sign on $(p_0, p_1)$ as desired.
\end{proof}

\section*{Acknowledgment}
The author is grateful to Slava Sedykh for correcting some computations in the proof of Lemma \ref{lem:osculate}, and several other comments to improve the exposition of this work. Thanks also to Masaaki Umehara for helpful communications.
\bibliographystyle{abbrv}
\bibliography{references}

\end{document}